\documentclass[12pt,a4paper]{article}
\usepackage{amsmath}  
\usepackage{amssymb}  
\usepackage{amsthm}   
\usepackage{hyperref}
\hypersetup{
breaklinks,
colorlinks=true,
linkcolor=blue,
pdfborder = 0 0 0,
citecolor=blue,
urlcolor=blue}
\newcommand\blfootnote[1]{%
  \begingroup
  \renewcommand\thefootnote{}\footnote{#1}%
  \addtocounter{footnote}{-1}%
  \endgroup
}
\numberwithin{equation}{section}

\makeatletter\@twosidetrue\makeatother

\newtheorem{theorem}{Theorem}[section]
\newtheorem{lemma}[theorem]{Lemma}
\newtheorem{corollary}[theorem]{Corollary}
\newtheorem{proposition}[theorem]{Proposition}

\newtheorem{definition}[theorem]{Definition}
\newtheorem{example}[theorem]{Example}

\newtheorem{remark}[theorem]{Remark}
\newtheorem{state}[theorem]{Assumption}

\title{\Large\textbf{On the solvability of a two-dimensional Ventcel problem with variable coefficients}}

\author{\normalsize Antonio Greco$^*$ and Giuseppe Viglialoro\\ \\
{\normalsize\centering{Dipartimento di Matematica e Informatica}}\\ {\normalsize\centering{Universit\`a di Cagliari. Via Ospedale 72.
09124, Cagliari (Italy)}}}

\date{}

\begin{document}

\belowdisplayshortskip = \belowdisplayskip
\paperwidth=198 true mm
\paperheight=297 true mm
\pdfpagewidth=198 true mm
\pdfpageheight=297 true mm

\maketitle

\pdfbookmark[1]{Abstract}{Abstract}
\begin{abstract}
This paper deals with the following mixed boundary value problem 
\begin{equation}\label{ProblemAbstract}
\tag{$\Diamond$}
\begin{cases}
-\Delta u = f &\mbox{in $\Omega$,}
\\
u = \varphi &\mbox{on $\Gamma_{\! D}$,}
\\
u_\nu - a_2 \, \Delta_{\tau \,} u + a_0 \, u = g
&\mbox{on $\Gamma_{\! \nu}$,}
\end{cases}
\end{equation}
where $\Omega$ is some bounded domain of $\mathbb{R}^2$ with $\partial \Omega=\Gamma_{\!D}\cup \Gamma_{\! \nu}$, $\nu$ indicating the normal unit vector to $\Gamma_{\! \nu}$ and $\Delta_\tau$ the Laplace--Beltrami operator along~$\Gamma_{\! \nu}$. Additionally, $f(\bf x)$, $\varphi(\bf x)$, $a_2(\bf x)$, $a_0(\bf x)$ and $g(\bf x)$ are convenient functions defined on $\Omega$, $\Gamma_{\!D}$ and $\Gamma_{\! \nu}$, and ${\bf x} = (x,y)$ denotes a two-dimensional array. Under suitable assumptions on the data, we first give the definition of a weak solution $u$ to the problem and then we prove that it is uniquely solvable. Further, in the spirit of \cite{ViglialoroGonzalezMurciaIJCM}, we consider a particular case of \eqref{ProblemAbstract} arising in real-world applications: we discuss the resulting model and provide an explicit solution. 
\end{abstract}





{\blfootnote{{$^*$Corresponding author: greco@unica.it}}}
\section{Introduction and motivations}\label{SectionIntro}
An increasing interest is currently being devoted to the Ventcel
problem, which originates in the pioneering papers by Ventcel
\cite{V1956,V1959} and Feller \cite{F1952,F1957}, and has been
recently investigated by several authors (see, for instance,
\cite{BNDHV,CLNV,KCDQ,Nicaise-Li-Mazzucato} and the references
therein).

The Ventcel problem is a boundary-value problem of mixed type whose
peculiarity relies in the fact that one of the boundary conditions
involves a second-order differentiation of the solution along a
portion of the boundary; more specifically, the order of the intrinsic
derivative along this part equals the one of the interior differential
operator. This non-standard relation arises when modeling physical
phenomena: from heat conduction processes with a heat source on the
boundary to fluid-structure interaction problems (the reader can find
a deep overview on these topics in \cite[Section 1]{KCDQ}, and in
other references therein listed). Furthermore, as far as our precise
interest is concerned, the mentioned singular boundary condition also
appears when looking for the equilibrium, in planar elasticity, of a
prestressed membrane whose boundary is composed of rigid and cable
elements. In particular, the membrane-cable interface equilibrium
results in a Ventcel-type boundary condition; this is
physically and mathematically discussed in
\cite{ViglialoroGonzalezMurciaIJCM}. More exactly, the mixed boundary
value problem modeling the equilibrium of the membrane is addressed in
\cite{ViglialoroGonzalezMurciaIJCM} by means of a numerical strategy,
leaving the theoretical questions concerning the existence and
uniqueness of solutions open. This challenge was the initial motivation of our
investigation, whose aim is proving such well-posedness (so giving a
more complete picture on the issue) and also considering more general
cases than that in \cite{ViglialoroGonzalezMurciaIJCM} and not
included in the known literature. In order to clarify their
complementary roles, Appendix $\S$\ref{AppendixSection} at the end of
the paper includes some details correlating this present article and
\cite{ViglialoroGonzalezMurciaIJCM} as well as an example of an
explicit resolution.

\subsection{Formulation of the problem: some known results}
Following the notation in~\cite{BNDHV} and
\cite{Nicaise-Li-Mazzucato}, for some bounded domain $\Omega \subset
\mathbb{R}^{d}$, $d\geq 2$, with Lipschitz boundary $\partial \Omega$,
let $\Gamma_{\! \nu}$ be an open subset of $\partial \Omega$ having
positive measure and such that $\Gamma_{\! D}=\partial \Omega
\setminus \Gamma_{\! \nu}$ also has positive measure. For a convenient
function $f=f({\bf x})$, ${\bf x}\in \Omega$, and boundary data $\varphi=\varphi({\bf x})$ and $g=g({\bf x})$, ${\bf x} \in \Gamma_\nu$, as well as for variable coefficients $a_2=a_2({\bf x})$ and $a_0=a_0({\bf x})$, with ${\bf x}\in \Gamma_\nu$, we are interested in the Ventcel mixed boundary value problem
	\begin{equation}\label{inhomogeneous}
	\begin{cases}
	-\Delta u = f &\mbox{in $\Omega$;}
	\\
	u = \varphi &\mbox{on $\Gamma_{\! D}$;}
	\\
	u_\nu - a_2 \, \Delta_{\tau \,} u + a_0 \, u = g
	&\mbox{on $\Gamma_{\! \nu}$,}
	\end{cases}
	\end{equation}
	where $\Delta_{\tau \,}$ is the Laplace--Beltrami operator on $\Gamma_{\! \nu}$ and the subscript  $(\cdot)_\nu$ indicates the outward normal derivative on $\Gamma_{\! \nu}$.
	
In order to frame our contribution in terms of what is already available in the literature, we mention some known results concerning problem \eqref{inhomogeneous}:
\begin{itemize}
\item [-] For $\Gamma_{\!D}=\emptyset$, $\Omega \subset \mathbb{R}^{d}$,
  with $d\geq 2$, $a_0=\alpha>0$ (constant), $a_2=-\beta$ (constant)
  and $g\equiv 0$, well-posedness has been established independently of the sign of $\beta$: see \cite{BNDHV} (the notation is taken from there).
\item [-] For $\Gamma_{\!D}=\emptyset$, $\Omega \subset \mathbb{R}^{d}$,
  with $d\geq 2$, $a_0=\alpha>0$ (constant) and $a_2=\beta>0$
  (constant), regularity results and finite element analysis are
  presented in \cite{KCDQ}.
\item [-] For $\Omega \subset \mathbb{R}^{d}$, with $d=3$, $a_0=0$,
  $a_2=1$ (constant) and $\varphi\equiv 0$, regularity results and a
  priori error analysis in polyhedral domains are obtained in \cite{Nicaise-Li-Mazzucato}.
\end{itemize}
\subsection{Presentation of the main result: novelties and difficulties}
In accordance with the previous premises, we intend to enhance the mathematical theory tied to problem \eqref{inhomogeneous} by extending some of the results  summarized in the above items to the case where the coefficients associated to the Laplace--Beltrami operator effectively depend on the spatial variable. Our  bibliographic research did not show any result in this direction so that, since taking $a_2$ explicitly varying with ${\bf x}\in \Omega$ leads to certain technical complexities, we understand that the corresponding analysis represents a novelty in the field deserving to be studied.  Such an analysis deals with existence and uniqueness of a weak solution to problem~(\ref{inhomogeneous}) under suitable assumptions on the given functions $f,g,\varphi$ and the spatial-dependent coefficients $a_2$ and $a_0$. In particular, to formulate the  main claim of our work we first have to set the following precise assumption:
\begin{state}\label{MainAssumption}
~
\begin{enumerate}
\renewcommand\labelenumi{\rm(\arabic{enumi})}
\item $\Omega$ is a bounded Lipschitz domain (an open, connected set) in the plane, whose boundary $\partial \Omega$ properly contains a finite number $N \ge 1$ of simple $C^1$-curves $\Gamma_i$, $i = 1, \dots, N$, satisfying $\Gamma_i \cap \Gamma_j = \emptyset$ for $i \ne j$.
\item\label{parametric}  By a $C^1$-curve~$\Gamma_i$ we mean the trace of a function $r = r_i(s)$ belonging to the class $C^1([0,L_i], \, \mathbb R^2)$ and satisfying $|r'_i(s)| = 1$ in the interval $[0,L_i]$. Thus, $L_i > 0$ is the length of~$\Gamma_i$, and the tangent unit vector $\tau = \partial/\partial s$, as well as the outward unit vector~$\nu$ extend by continuity from the interior of\/~$\Gamma_i$ to its endpoints.
\item\label{simple} Each curve $\Gamma_i$ is simple in the sense that the equality $r(s_1) = r(s_2)$ may hold for $s_1 < s_2$ only if $s_1 = 0$ and $s_2 = L$. We allow the case when $r_i(0) = r_i(L_i)$, i.e., $\Gamma_i$ is a closed curve, but in such a case we require $r'_i(0)= r'_i(L_i)$: thus, $\tau$ and~$\nu$ have a unique extension to the (coinciding) endpoints.
\item The notation $\Gamma_{\! i}^\circ$ stands for the interior of\/~$\Gamma_i$, which is defined as follows: $\Gamma_{\! i}^\circ = \Gamma_{\! i} \setminus \{\, r_i(0), \, r_i(L_i) \,\}$ if $r_i(0) \ne r_i(L_i)$, and\/ $\Gamma_{\! i}^\circ = \Gamma_{\! i}$ if $r_i(0) = r_i(L_i)$. We also define the open subset\/ $\Gamma_{\! \nu} = \bigcup_{i = 1}^N \Gamma_{\! i}^\circ \subset \partial \Omega$, and we assume that the closed subset\/ $\Gamma_{\! D} = \partial \Omega \setminus \Gamma_{\! \nu}$ has a positive length.
\end{enumerate}
\end{state}
\begin{example}\label{ring-example}
A domain fulfilling the above structure is, for instance, the annulus $\Omega = B_{R_1}(0) \setminus \overline B_{R_0}(0)$, with $0 < R_0 < R_1$. We may let\/ $\Gamma_{\! \nu} = \partial B_{R_0}(0)$ and\/ $\Gamma_{\! D} = \partial B_{R_1}(0)$, as well as $\Gamma_{\! \nu} = \partial B_{R_1}(0)$ and\/ $\Gamma_{\! D} = \partial B_{R_0}(0)$: the last  case is considered in~\cite[Section~2.1.2]{BNDHV}. Another example is the square\/ $\Omega = (0,1) \times (1,2)$, $\Gamma_i = [0,1] \times \{\, i \,\}$, $i = 1,2$ and\/ $\Gamma_{\! D} = \{\, 0,1 \,\} \times [1,2]$.
\end{example}

\goodbreak In view of the above, our result is hence summarized in this
\begin{theorem}\label{Maintheorem}
Let Assumption \ref{MainAssumption} be complied and $\varphi \colon \Gamma_{\! D} \to \mathbb R$ satisfy a uniform Lipschitz condition. Then for any $f \in L^2(\Omega)$, $g \in L^2(\Gamma_{\! \nu})$, $0<a_2 \in W^{1,\infty}(\Gamma_{\! \nu})$ and $0\leq a_0 \in L^\infty(\Gamma_{\! \nu})$, problem \eqref{inhomogeneous} admits a unique weak solution.
\end{theorem}
\begin{remark}\label{RemarkLaplOperOnaCurve}
Taking Assumption \ref{MainAssumption}~(\ref{parametric}) into account, the derivative $(u(r_i(s)))'=d u(r_i(s))/d s$ of a smooth function~$u$ is denoted with\/ $\nabla_{\! \tau\,} u$, the intrinsic gradient of $u$ on $\Gamma_i$, and the second derivative $(u(r_i(s)))'' =d^2 u(r_i(s))/d s^2$ with $\Delta_{\tau \,} u$, the Laplace--Beltrami operator on\/ $\Gamma_i$. For shortness, we also write $u_s$ and~$u_{ss}$ in place of\/ $\nabla_{\! \tau\,} u$ and\/ $\Delta_{\tau \,} u$, respectively.
\end{remark}
As to the first cornerstone issue (see Definition \ref{def:inhomogeneous} given below), taking $a_2$ in a suitable function space gives the third condition in~(\ref{inhomogeneous}) a variational structure: namely, the condition becomes
\begin{equation*}
u_\nu-\nabla_{\!\tau\,} (a_2 \, \nabla_{\!\tau\,} u)+\nabla_{\!\tau\,} a_2 \, \nabla_{\!\tau\,} u  +a_0 \, u=g\quad \textrm{on}\quad \Gamma_\nu.
\end{equation*}
Secondly, when $a_2$ is constant, the coercivity of the bilinear form associated to the definition of a weak solution is essentially automatic, while in the non-constant case coercivity is gained (as shown in Lemma \ref{LemmaBilinearFormAdHoc}) by adding a convenient term to both sides of the functional equation resulting from the weak formulation of problem \eqref{inhomogeneous} so to obtain a compact resolvent operator, successively used to invoke the Fredholm alternative. The weak formulation of the problem also relies on the existence of a sufficient regular lifting of the boundary data to the whole domain: in order to make this paper self-contained, we give details in Lemma~\ref{lifting}.

\subsection{Non-local interpretation of the problem}
To the sake of scientific completeness, we shortly mention that
problem~(\ref{inhomogeneous}) can be rephrased in terms of a suitable
non-local operator. The idea is that if the domain~$\Omega$ and the
boundary data $\varphi,\psi$ are sufficiently smooth, then the
solution $u$ of the Dirichlet problem
\begin{equation}\label{Dirichlet}
\begin{cases}
-\Delta u = f &\mbox{in $\Omega$;}
\\
u = \varphi &\mbox{on $\Gamma_{\! D}$;}
\\
u = \psi &\mbox{on $\Gamma_{\! \nu}$,}
\end{cases}
\end{equation}
is uniquely determined, and differentiable on~$\Gamma_\nu$. If the boundary value~$\varphi$ on~$\Gamma_{\! D}$ is kept fixed, and $\psi$ is let vary, then the
outward derivative $u_\nu$ along~$\Gamma_\nu$ can be thought of as the outcome of an
operator: namely the operator $L \colon \psi \mapsto u_\nu$, usually called
\textit{the Dirichlet-to-Neumann operator} (a related operator is the \textit{Steklov-Poincar\'e operator} defined in \cite[pp.~3-4]{QV}). The Dirichlet-to-Neumann operator is considered in detail in~\cite{CS} in the case when the bounded domain
$\Omega \subset \mathbb R^2$ is replaced by the half-space in $\mathbb
R^{d-1} \times (0,+\infty) \subset \mathbb R^d$, $d \ge 2$, and $\Gamma_{\! D} = \emptyset$. In the present case,
instead, the operator~$L$ acts on functions defined on the bounded,
possibly non-straight and disconnected curve~$\Gamma_{\! \nu}$. Such
an operator is non-local in the sense that any modification of the
given function~$\psi$ in a small neighborhood of any point ${\bf x}_0 \in
\Gamma_{\! \nu}$ implies a consequent modification of the
outcome~$u_\nu$ on the whole~$\Gamma_{\! \nu}$. Under convenient
assumptions, the Ventcel problem~(\ref{inhomogeneous}) is equivalent
to the single equation
\begin{equation}\label{single}
L\psi - a_2 \, \Delta_{\tau \,} \psi + a_0 \, \psi = g
\quad
\mbox{on $\Gamma_{\! \nu}$}
.
\end{equation}
Indeed, if $u$ is any solution of problem~(\ref{inhomogeneous}), then
its trace $\psi = u|_{\Gamma_{\! \nu}}$ solves equation~(\ref{single})
because $L\psi = u_\nu$. Conversely, if $\psi$ is any solution of
equation~(\ref{single}), then the solution~$u$ of the Dirichlet
problem~(\ref{Dirichlet}), which is used in the definition of~$L$, is
also a solution of the Ventcel problem~(\ref{inhomogeneous}). The
peculiarity of equation~(\ref{single}) lies in the fact that the
non-local operator~$L$ and the (local) Laplace--Beltrami operator are
compared to each other. This approach will not be developed in the
present paper; nonetheless, the transformation of a Ventcel problem into a single
non-local equation is investigated in detail
in~\cite[Section~2.2.1]{BNDHV}.
\section{Poincar\'e-type inequalities}
In this section we establish some estimates that resemble the well-known Poincar\'e inequality, and play a key role in the subsequent proof of uniqueness of the weak solution to problem~(\ref{inhomogeneous}). In particular, in Lemma~\ref{NormOnVGeneral} we also prove an interesting Poincar\'e inequality for the composite Sobolev space $V_0$ defined in~\eqref{DefiSpaceFunctionV}.
\begin{lemma}[Poincar\'e-type inequality in a bounded interval]\label{LemmaPTI}
Let\/ $(a,b)$ be a bounded interval on the real line, and let $v \in H^1((a,b))$. Since $v$ has a continuous extension to the closed interval\/ $[a,b]$, we may define
$$
v_0 = \min_{[a,b]} v,
\qquad
v_1 = \max_{[a,b]} v,
\qquad
\tilde I
=
\{\, s \in [a,b]
\mid
v_0 < v(s) < v_1 \}
.
$$
We claim that
\begin{equation}\label{PTI}
\int_{\tilde I}
\big(v(s) - v_0\big)^{\! 2} \, ds
\le
|\tilde I|^2
\int_{\tilde I} (v'(s))^2 \, ds
,
\end{equation}
where $|\tilde I|$ denotes the Lebesgue measure of\/~$\tilde I$.
\end{lemma}

\begin{proof}
If $\tilde I = \emptyset$, i.e., if $v$ is constant, then (\ref{PTI}) obviously holds. Otherwise we argue as follows. Let $s_0 \in [a,b]$ be any point such that $v(s) = v_0$. By the fundamental theorem of calculus we get the estimate
$$
v(s) - v_0
\le
\int_{s_0}^s |v'(t)| \, dt
\le
\Vert v' \Vert_{L^1((a,b))}
.
$$%
By~\cite[Lemma~7.7]{GT} we have $v' = 0$ almost everywhere in~$(a,b)
\setminus \tilde I$, and hence we may replace $L^1((a,b))$ with
$L^1(\tilde I)$ in the inequality above, thus getting
$$
v(s) - \inf\limits_{(a,b)} v
\le
\Vert v' \Vert_{L^1(\tilde I)}
.
$$%
An application of the Cauchy-Schwarz inequality on the set~$\tilde I$
yields $\Vert v' \Vert_{L^1(\tilde I)} \le \Vert 1 \Vert_{L^2(\tilde I)} \,
\Vert v' \Vert_{L^2(\tilde I)} = |\tilde I|^\frac12 \, \Vert v'
\Vert_{L^2(\tilde I)}$. This and the preceding inequality yield the
pointwise estimate
$$
\big(v(s) - v_0\big)^{\! 2}
\le
|\tilde I| \, \Vert v' \Vert^2_{L^2(\tilde I)}
,
$$
and (\ref{PTI}) follows by integration on~$\tilde I$.
\end{proof}

\begin{corollary}[Poincar\'e-type inequality on a simple $C^1$-curve]
Let\/ $\Gamma$ be the trace of a function $r = r(s)$ belonging to the class $C^1([0,L],\mathbb R^d)$, $L \in (0,+\infty)$, $d \ge 2$, and satisfying $|r'(s)| = 1$ in $[0,L]$. Assume that\/ $\Gamma$ is a simple curve satisfying Assumption~\ref{MainAssumption} (\ref{simple}). Since each $w\in H^1(\Gamma)$ has a continuous representative, we may define
$$
w_0 = \min\limits_\Gamma w,
\qquad
w_1 = \max\limits_\Gamma w,
\qquad
\tilde \Gamma = \{\, {\bf x} \in \Gamma \mid w_0 < w({\bf x}) < w_1 \, \}
.
$$
We claim that
\begin{equation}\label{PoincPesataLemmaBISAntonio}
\int_{\tilde \Gamma}
\big(w(r(s)) - w_0\big)^{\! 2} \, ds
\le
|\tilde \Gamma|^2
\int_{\tilde \Gamma}
|\nabla_{\!\tau\,} w(r(s))|^2 \, ds
,
\end{equation}
where $|\tilde \Gamma|$ denotes the Hausdorff measure of\/~$\tilde \Gamma$.
\end{corollary}

\begin{proof}
Let $(a,b) = (0,L)$ and $v(s) = w(r(s))$. Since $|r'(s)| = 1$, the Hausdorff measure (or total length) of $\tilde \Gamma$ equals the Lebesgue measure of the set $\tilde I$ in Lemma~\ref{LemmaPTI}. Furthermore, by the definition of the intrinsic gradient $\nabla_{\! \tau \,} w$ (see Remark \ref{RemarkLaplOperOnaCurve}) we have $|\nabla_{\! \tau \,} w(s)| = |v'(s)|$, and the conclusion is accomplished through inequality~(\ref{PTI}).
\end{proof}
In order to define an appropriate functional setting where facing problem \eqref{inhomogeneous}, it is convenient to introduce the following vector spaces on a domain~$\Omega$ satisfying Assumption \ref{MainAssumption}:
\begin{equation*}
V
=
\{\,
v \in H^1(\Omega)
\mid
v_{|_{\Gamma_{\! \nu}}} \in H^1(\Gamma_{\! \nu})
\,\}
,
\end{equation*}
and
\begin{equation}\label{DefiSpaceFunctionV}
V_0
=
\{\,
v \in H^1(\Omega)
\mid
v_{|_{\Gamma_{\! \nu}}} \in H^1_0(\Gamma_{\! \nu}),\
v_{|_{\Gamma_{\! D}}} = 0
\,\},
\end{equation}
both endowed with the  norm
\begin{equation}\label{NormOnVGeneral} 
\lVert v \rVert^2_{V}
:=
\int_\Omega |Du|^2 \, d {\bf x}
+\int_\Omega u^2 \, d {\bf x}+
\int_{\Gamma_\nu} |\nabla_{\!\tau\,} u|^2 \, ds+\int_{\Gamma_\nu}u^2 \,ds
\end{equation}
where $Du=Du(x,y)=(u_x(x,y),u_y(x,y))$ is the gradient of $u$. We have:
\begin{lemma}[Poincar\'e inequality for $V_0$]\label{PoincPesataLemma}
Let\/ $\Omega$ and\/ $\Gamma_{\! \nu}$ be as in Assumption \ref{MainAssumption}, and
let\/ $V_0$ be given by\/~{\rm(\ref{DefiSpaceFunctionV})}. There exists a
constant $L > 0$ such that the inequality $\Vert v \Vert_{L^2(\Gamma_\nu)} \le L \, \Vert Dv \Vert_{L^2(\Omega)}$ holds for every $v \in V_0$.
\end{lemma}
\begin{proof}
By the usual Poincar\'e inequality (see \cite[Theorem 7.91]{Salsa}), there exists a constant~$C_P$ such that $\Vert v \Vert_{H^1(\Omega)} \le C_P \, \Vert Dv \Vert_{L^2(\Omega)}$. Furthermore, by the trace inequality \cite[Theorem~7.82]{Salsa} we also have $\Vert v \Vert_{L^2(\Gamma_\nu)} \le C \, \Vert v \Vert_{H^1(\Omega)}$ for a convenient~$C$, and the conclusion follows.
\end{proof}

\begin{remark}
Some care is needed when dealing with the Poincar\'e inequality in the space~$V_0$: in particular, it is not to be expected that\/ $\Vert w \Vert_{L^2(\Gamma_{\!\nu})} \le C \, \Vert \nabla_{\!\tau\,} w \Vert_{L^2(\Gamma_{\!\nu})}$ for some constant~$C$ and for all $w \in V_0$. To construct a counterexample, let us consider the annulus\/ $\Omega = B_{R_1}(0) \setminus \overline B_{R_0}(0)$, with\/ $\Gamma_{\! \nu} = \partial B_{R_1}(0)$ and\/ $\Gamma_{\! D} = \partial B_{R_0}(0)$ as in Example~\ref{ring-example}. The radial function $w({\bf x}) = (|{\bf x}| - R_0)/(R_1 - R_0)$ belongs to~$V_0$ and satisfies $\Vert w \Vert_{L^2(\Gamma_{\!\nu})} = |\Gamma_{\!1}|^\frac12$ as well as $\Vert \nabla_{\!\tau\,} w \Vert_{L^2(\Gamma_{\!\nu})} = 0$. Hence a constant~$C$ as above cannot exist.
\end{remark}
\section{Definition of  weak solutions}
As a consequence of Lemma \ref{PoincPesataLemma}, and by means of the two-dimensional Poncar\'e inequality, the following norm on $V_0$ (recall the definition in \eqref{DefiSpaceFunctionV}) 
\begin{equation}\label{NormOnV0}
\lVert v \rVert^2_{V_0}
:=
\int_\Omega |Du|^2 \, d {\bf x}
+
\int_{\Gamma_\nu} |\nabla_{\!\tau\,} u|^2 \, ds
\end{equation}
is equivalent to  \eqref{NormOnVGeneral}, and standard procedures show that $(V_0,\lVert  \cdot \rVert_{V_0})$ is a weighted Sobolev space (see, for instance,  \cite{BNDHV,CLNV,Nicaise-Li-Mazzucato}).
On the other hand, we also need that the Dirichlet datum~$\varphi \colon \Gamma_{\! D} \to \mathbb R$ in problem~(\ref{inhomogeneous}) has a lifting $\tilde \varphi \in V$, i.e., we need that $\varphi$ is the trace on~$\Gamma_{\! D}$ of some $\tilde \varphi \in V$. Before proceeding further, we give a sufficient condition for~$\varphi$ to satisfy such a requirement:
\begin{lemma}[Existence of a lifting]\label{lifting}
Let Assumption \ref{MainAssumption} be complied. If $\varphi \colon \Gamma_{\! D} \to \mathbb R$ satisfies a uniform Lipschitz condition, then there exists $\tilde \varphi \in V$ such that $\varphi = \tilde \varphi\vert_{\Gamma_{\! D}}$.
\end{lemma}

\begin{proof}
By McShane's lemma \cite[Theorem~1]{McShane}, there exists a uniformly Lipschitz function $\tilde \varphi \colon \mathbb R^2 \to \mathbb R$ coinciding with~$\varphi$ on~$\Gamma_{\! D}$. In particular, the restriction $\tilde \varphi|_{\Gamma_{\!\nu}}$ belongs to the Sobolev space $W^{1,\infty}(\Gamma_{\!\nu}) \subset H^1(\Gamma_{\!\nu})$, and therefore $\tilde \varphi \in V$.
\end{proof}

\subsection{The inhomogeneous problem }
Following \cite[(10), p.~314]{Evans-2010-PDEs} and \cite[(8.3)]{GT}, we allow $f$ and~$g$ be given by $f = f_1 + \mathop{\rm div}{\bf f}_2$ and $g = g_1 + \nabla_{\!\tau\,} g_2$. Such equalities are intended in the weak sense: namely, every choice of $f_1 \in L^2(\Omega, \allowbreak \mathbb R)$, ${\bf f}_2 \in L^2(\Omega,\mathbb R^2)$ and $g_1,g_2 \in L^2(\Gamma_{\!\nu})$ gives rise to two linear, continuous operators $L_f,L_g$ acting on the function space~$V_0$ through
\begin{align*}
L_f \colon w
&\mapsto
\int_\Omega
(f_1 \, w - {\bf f}_2 \cdot Dw)
\,
d {\bf x}
,
\qquad
L_g \colon w
\mapsto
\int_{\Gamma_{\! \nu}}
(g_1 \, w - g_2 \, \nabla_{\!\tau\,} w)
\,
ds
.
\end{align*}
For shortness, we let ${\cal L}w$ be the sum of the two operators above:
$$
{\cal L}w
=
L_f + L_g.
$$
If the boundary datum~$\varphi \colon \Gamma_{\! D} \to \mathbb R$ has
a lifting $\tilde \varphi \in V$, we may define the function space
$$
V_{\!\varphi}
=
\{\,
u \in H^1(\Omega)
\mid
u - \tilde \varphi \in V_0
\,\}
$$%
and give the definition of  weak solution of problem~(\ref{inhomogeneous}):
\begin{definition}[Weak solution of the inhomogeneous problem]\label{def:inhomogeneous}
Let Assumption \ref{MainAssumption} be complied and let $f_1 \in L^2(\Omega,\mathbb R)$, ${\bf f}_2 \in L^2(\Omega,\mathbb R^2)$, $g_1,g_2 \in L^2(\Gamma_{\! \nu})$, $a_2 \in W^{1,\infty}(\Gamma_{\! \nu})$ and $a_0 \in L^\infty(\Gamma_{\!\nu})$. Suppose that $\varphi \colon \Gamma_{\! D} \to \mathbb R$ has a lifting $\tilde \varphi \in V$. A weak solution of problem~(\ref{inhomogeneous}) is a function $u \in V_{\!\varphi}$ such that for every $w \in V_0$ the following equality holds:
\begin{equation}\label{weak}
\int_\Omega
Du \cdot Dw \, d {\bf x}
+
\int_{\Gamma_{\! \nu}}
\Big(
\big(
a_2 \, \nabla_{\! \tau \,} w
+
w \, \nabla_{\! \tau \,} a_2
\big)
\, \nabla_{\! \tau \,} u
+
a_0 \, uw
\Big)
ds
=
{\cal L}w
.
\end{equation}
\end{definition}
The definition is motivated by the following property:
\begin{proposition}
Let Assumption \ref{MainAssumption} be complied and $f \in C^0(\Omega)$, $g \in C^0(\Gamma_{\! \nu})$, $a_2 \in C^1(\Gamma_{\! \nu})$, $a_0 \in C^0(\Gamma_{\! \nu})$ and $\varphi \in C^1(\Gamma_{\! D})$. Any function $u \in C^2(\overline \Omega)$ is a weak solution of problem\/~{\rm(\ref{inhomogeneous})} if and only if the three equalities in\/~{\rm(\ref{inhomogeneous})} hold pointwise.
\end{proposition}
\begin{proof}
Step 1. Let us prepare an identity to be used afterwards. By the product rule, for every $w \in \allowbreak  V_0$ we have $\nabla_{\! \tau} \big((a_2 \, w) \, \nabla_{\! \tau \,} u\big) = \nabla_{\! \tau} (a_2 \, w) \allowbreak\, \nabla_{\! \tau \,} u \allowbreak + a_2 \, w \, \Delta_{\tau \,} u$ in the weak sense on~$\Gamma_i$, $i = 1, \dots, N$, as well as $\nabla_{\! \tau} (a_2 \, w) = a_2 \, \nabla_{\! \tau \,} w + w \, \nabla_{\! \tau \,} a_2$; hence
\begin{align}
\nonumber
\int_{\Gamma_i}
\big(
a_2 \, \nabla_{\! \tau \,} w
+
w \, \nabla_{\! \tau \,} a_2
\big)
\,
\nabla_{\! \tau \,} u
\,
ds
&=
\int_{\Gamma_i}
\nabla_{\! \tau} (a_2 \, w)
\,
\nabla_{\! \tau \,} u
\,
ds
\\
\noalign{\medskip}
\label{observe}
&=
\int_{\Gamma_i}
\nabla_{\! \tau}
\big((a_2 \, w) \, \nabla_{\! \tau \,} u\big)
\,
ds
-
\int_{\Gamma_i}
a_2 \, w \, \Delta_{\tau \,} u
\,
ds
.
\end{align}
Furthermore, by the fundamental theorem of calculus we also have
\begin{align*}
\int_{\Gamma_i}
& \nabla_{\! \tau}
\big((a_2 \, w) \, \nabla_{\! \tau \,} u\big)
\,
ds
= \\ & a_2(r_i(L_i)) \, w (r_i(L_i)) \, \nabla_{\! \tau \,} u(r_i(L_i))-a_2(r_i(0)) \, w (r_i(0)) \, \nabla_{\! \tau \,} u(r_i(0))
\end{align*}
for every $i = 1, \dots, N$. If $r_i(0) = r_i(L_i)$, then the right-hand side obviously vanishes. If, instead, $r_i(0) \ne r_i(L_i)$, then $w(r_i(0)) = w(r_i(L_i)) = 0$ because $w \in H^1_0(\Gamma_i)$. Consequently, we arrive at
$$
\int_{\Gamma_i}
\nabla_{\! \tau}
\big((a_2 \, w) \, \nabla_{\! \tau \,} u\big)
\,
ds
=
0\
\mbox{for every $i = 1, \dots, N$.}
$$%
Plugging this into~(\ref{observe}), and summing over~$i$ leads to the identity
\begin{equation}\label{identity}
\int_{\Gamma_{\! \nu}}
\big(
a_2 \, \nabla_{\! \tau \,} w
+
w \, \nabla_{\! \tau \,} a_2
\big)
\,
\nabla_{\! \tau \,} u
\,
ds
=
-
\int_{\Gamma_{\! \nu}}
a_2 \, w \, \Delta_{\tau \,} u
\,
ds
.
\end{equation}
Step 2. Suppose that $u \in C^2(\overline \Omega)$ is a weak solution of problem~(\ref{inhomogeneous}). Then by the definition of a weak solution we have $u \in V_{\!\varphi}$ and therefore $u = \varphi$ pointwise on~$\Gamma_{\!D}$ (cf.~\cite[Corollary~1.5.1.6]{Grisvard}). Moreover, since $C^1_c(\Omega) \subset V_0$, we may take any $\psi \in C^1_c(\Omega)$ and let $w = \psi$ in~(\ref{weak}), thus obtaining
$$
\int_\Omega
Du \cdot D\psi \, d {\bf x}
=
\int_\Omega
f\psi
\,
d {\bf x}
\quad
\forall \psi \in C^1_c(\Omega)
,
$$%
which implies
\begin{equation}\label{ptwOmega}
-\Delta u = f\ \mbox{in\/~$\Omega$}
\end{equation}
(see, for instance, \cite[Step~D, p.~293]{Brezis}). In order to prove the last equality in~\eqref{inhomogeneous}, we start from an arbitrary Lipschitz function $\eta \in W^{1,\infty}(\Gamma_{\!\nu}) \cap H^1_0(\Gamma_{\!\nu})$ and we define $w = \eta$ on~$\Gamma_{\!\nu}$, $w = 0$ on~$\Gamma_{\!D}$. Thus, $w$ is Lipschitz continuous on~$\partial \Omega$ and therefore, as mentioned in the proof of Lemma~\ref{lifting}, it has a Lipschitz continuous lifting to the whole plane~$\mathbb R^2$. We denote such a lifting again by~$w$, for simplicity, and observe that $w \in V_0$. Multiplying both sides of~(\ref{ptwOmega}) by $w$, and integrating by parts we obtain
\begin{equation}\label{partsOmega}
\int_\Omega
Du \cdot Dw \, d {\bf x}
-
\int_{\Gamma_{\! \nu}}
wu_\nu
\,
ds
=
\int_\Omega
fw
\,
d {\bf x}
.
\end{equation}
Recall that identity~(\ref{weak}) holds for all $w \in V_0$ by assumption. By comparing (\ref{weak}) with the equality above, we deduce that
\begin{equation}\label{intermediate}
\int_{\Gamma_{\! \nu}}
\Big(
wu_\nu
+
\big(
a_2 \, \nabla_{\! \tau \,} w
+
w \, \nabla_{\! \tau \,} a_2
\big)
\, \nabla_{\! \tau \,} u
+
a_0 \, uw
\Big)
ds
=
\int_{\Gamma_{\! \nu}}
gw
\,
ds
.
\end{equation}
Since $w = \eta$ on~$\Gamma_{\!\nu}$, this and the identity~(\ref{identity}) yield
\begin{equation}\label{integrated}
\int_{\Gamma_{\! \nu}}
\Big(
\eta u_\nu
-
a_2 \, \eta \, \Delta_{\tau \,} u
+
a_0 \, u\eta
\Big)
ds
=
\int_{\Gamma_{\! \nu}}
g\eta
\,
ds
,
\end{equation}
and the last equality in~\eqref{inhomogeneous} follows from the fundamental lemma of the calculus of variations (see \cite[Corollary~4.24]{Brezis}).

Step 3. Conversely, assume that a function $u \in C^2(\overline \Omega)$ satisfies the three equalities in~\eqref{inhomogeneous} pointwise. Then $u \in V_{\!\varphi}$ (see again~\cite[Corollary~1.5.1.6]{Grisvard}, or also Theorem~9.17 and the subsequent Remark~19 in~\cite{Brezis}). To proceed further, we multiply~(\ref{ptwOmega}) by an arbitrary test function $w \in V_0$ and integrate by parts, thus obtaining~(\ref{partsOmega}). On the other side, multiplying the last equality in~\eqref{inhomogeneous} by the same arbitrary~$w$, and integrating over~$\Gamma_{\! \nu}$ we get~(\ref{integrated}). This and the identity~(\ref{identity}) produce~(\ref{intermediate}), which summed to  (\ref{partsOmega}) leads to~(\ref{weak}), and the proof is complete.
\end{proof}
\subsection{The homogeneous problem }
With the aid of the lifting $\tilde \varphi \in V$ of the boundary
datum~$\varphi \colon \Gamma_{\! D} \to \mathbb R$, we may reduce the
inhomogeneous problem~(\ref{inhomogeneous}) to the homogeneous one
\begin{equation}\label{homogeneous}
\begin{cases}
-\Delta v = f_{\tilde\varphi} &\mbox{in $\Omega$;}
\\
v = 0 &\mbox{on $\Gamma_{\! D}$;}
\\
v_\nu - a_2 \, \Delta_{\tau \,} v + a_0 \, v = g_{\tilde\varphi}
&\mbox{on $\Gamma_{\! \nu}$,}
\end{cases}
\end{equation}
where $f_{\tilde\varphi},g_{\tilde\varphi}$ denote the linear, continuous operators on $V_0$, respectively defined as
\begin{align*}
f_{\tilde\varphi} \colon w
&\mapsto
L_f(w)
-
\int_\Omega
D \tilde\varphi
\cdot
Dw
\,
d {\bf x}
,
\\
\noalign{\medskip}
g_{\tilde\varphi} \colon w
&\mapsto
L_g(w)
-
\int_{\Gamma_{\! \nu}}
\Big(
a_0 \, \tilde \varphi \, w
+
\big(
a_2 \, \nabla_{\! \tau \,} w
+
w \, \nabla_{\! \tau \,} a_2
\big)
\, \nabla_{\! \tau \,} \tilde\varphi
\Big)
\,
ds
.
\end{align*}
\begin{remark}\label{RemarkFromInhToHom}
Any function $u \in V_{\tilde\varphi}$ is a weak solution of problem~(\ref{inhomogeneous}) if and only if the function $v = u - \tilde\varphi \in V_0$ is a weak solution of problem~(\ref{homogeneous}). For simplicity, in the sequel we will show that problem~(\ref{inhomogeneous}) is uniquely solvable by proving existence and uniqueness of a weak solution of problem~(\ref{homogeneous}).
\end{remark}
\section{\texorpdfstring{Unique solvability of the inhomogeneous problem: proof of Theorem \ref{Maintheorem}}{Unique solvability of the inhomogeneous problem: proof of Theorem \ref{Maintheorem}}}
In this section we will focus on the analysis concerning the existence and uniqueness of solutions to problem \eqref{homogeneous}. To this aim we first define, for any $a_2 \in W^{1,\infty}(\Gamma_{\! \nu})$ and $a_0 \in L^\infty(\Gamma_{\! \nu})$, some values which will be used throughout the main proofs. Precisely, we set
\begin{equation}\label{lambda-M}
\begin{cases}
\lambda_2
:=
\inf\limits_{\Gamma_{\! \nu}} a_2 > 0,
\quad \Lambda_2
:=
\sup\limits_{\Gamma_{\! \nu}} a_2,
\quad 
M
:=
\sup\limits_{\Gamma_{\! \nu}}
\left| \nabla_{\!\tau\,} a_2 \right|,\\
\noalign{\medskip}
 \lambda_0:=\inf\limits_{\Gamma_{\! \nu}} a_0 \ge 0, \quad \Lambda_0:=\sup\limits_{\Gamma_{\! \nu}} a_0.
\end{cases}
\end{equation}
\subsection{Uniqueness of the solution}
\begin{lemma}[Uniqueness of the solution of the homogeneous problem]\label{LemmaUniqueness}
Let Assumption \ref{MainAssumption} be complied. Suppose that\/ $0<a_2 \in W^{1,\infty}(\Gamma_{\! \nu})$ and\/ $0\leq a_0 \in L^\infty(\Gamma_{\! \nu})$.
 If\/ $v \in V_0$ satisfies
\begin{equation}\label{zero}
\int_\Omega
Dv \cdot Dw \, d {\bf x}
+
\int_{\Gamma_{\! \nu}}
\Big(
\big(
a_2 \, \nabla_{\! \tau \,} w
+
w \, \nabla_{\! \tau \,} a_2
\big)
\, \nabla_{\! \tau \,} v
+
a_0 \, vw
\Big)
\, ds
=
0
\end{equation}
for every $w \in V_0$, then $v$ vanishes almost everywhere in\/~$\Omega$.
\end{lemma}

\begin{proof}
Before proceeding further, observe that if $v \in V_0$
fulfills relation~(\ref{zero}) for every $w \in V_0$, and if furthermore
$\nabla_{\! \tau \,} v$ vanishes a.e.\ on~$\Gamma_{\!
  \nu}$, then
$$
\int_\Omega
Dv \cdot Dw \, d {\bf x}
+
\int_{\Gamma_{\! \nu}}
a_0 \, vw
\, ds
=
0
.
$$%
In such a case, choosing $w = v$ and recalling that $a_0 \ge 0$ on~$\Gamma_{\! \nu}$, it follows that $v = 0$. Hence, in order to prove the lemma, we consider an arbitrary $v \in V_0$ satisfying~(\ref{zero}) for every $w \in V_0$, and show that $\nabla_{\!\tau\,} v = 0$ a.e.\ on~$\Gamma_{\! \nu}$.
Suppose, by contradiction, that there exists a positive
integer $N_1 \le N$ such that the oscillation $\omega_i$ given by
$$
\omega_i
=
\max\limits_{\Gamma_{\! i}} v
-
\min\limits_{\Gamma_{\! i}} v
$$%
is positive if and only if $i \le N_1$ (see Figure~\ref{Fig.1}). Since $N_1$ is a finite number, there exists a positive~$\varepsilon_0$ such that $\omega_i \ge \varepsilon_0$ for every $i \le N_1$. Observe that at least one of the following two inequalities must hold:
$$
\mu_1 := \max\limits_{1 \le i \le N_1}
\max\limits_{\Gamma_{\! i}} v > 0;
\qquad
\min\limits_{1 \le i \le N_1}
\min\limits_{\Gamma_{\! i}} v < 0
.
$$
We consider the first case, the second one being analogous. Furthermore, without loss of generality we assume that $\mu_1$ is attained on $\Gamma_{\!1}$. By the definition of $\mu_1$ and~$\varepsilon_0$, for every $i = 1, \dots, N_1$ we have
$$
\min_{\Gamma_{\!i}} v
\le
\max_{\Gamma_{\!i}} v - \varepsilon_0
\le
\mu_1 - \varepsilon_0
$$%
and therefore
\begin{equation}\label{Gamma_1}
\min\limits_{\Gamma_{\! 1}} v
\le
\max\limits_{1 \le i \le N_1}
\min\limits_{\Gamma_{\! i}} v
=:
\mu_0
\le
\mu_1 - \varepsilon_0
<
\mu_1
=
\max\limits_{\Gamma_{\! 1}} v
.
\end{equation}
Arguing as in~\cite[Theorem~8.1]{GT} and in \cite[Theorem~3.2.1]{PS}, let us define $\mu_0^+ = \max\{\, \mu_0, 0 \,\}$ and choose a real number $k$ in the open interval $(\mu_0^+,\mu_1)$.
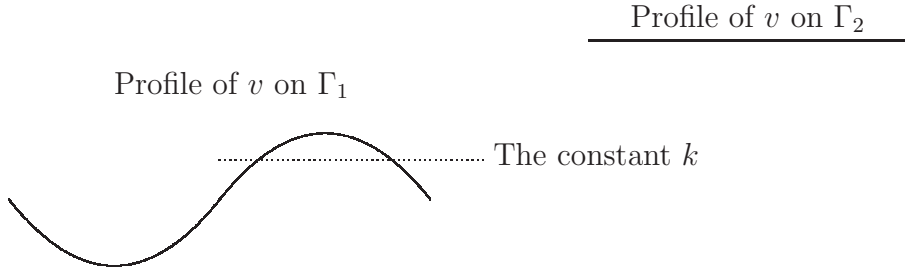
\begin{figure}[ht]
\centering
\begin{picture}(300,115)(0,20)
\qbezier(-20,50)(20,0)(60,50)
\put(20,90){Profile of $v$ on $\Gamma_1$}
\qbezier(60,50)(100,100)(140,50)
\qbezier[50](60,65)(110,65)(160,65)
\put(164,63){The constant $k$}
\put(216,115){Profile of $v$ on $\Gamma_2$}
\put(200,110){\line(1,0){120}}
\end{picture}
\caption{Sample profile of $v$ on the (rectified) curves $\Gamma_1$, $\Gamma_2$ ($N_1 = 1$)}
\label{Fig.1}
\end{figure}
Consider the function $w = (v - k)^+$. Since $k$ is positive, we have $w \in V_0$. Apart from a negligible set, the gradient~$Dw$ in~$\Omega$ is given by
\begin{equation}\label{Dw}
Dw({\bf x})
=
\begin{cases}
Dv({\bf x}), &\mbox{if $v({\bf x}) > k$,}
\\
\noalign{\medskip}
0, &\mbox{if $v({\bf x}) \le k$,}
\end{cases}
\end{equation}
and a similar representation holds for $\nabla_{\! \tau \,} w$
on~$\Gamma_{\! \nu}$. In particular, we have
$$
\nabla_{\! \tau \,} w \, \nabla_{\! \tau \,} v
=
|\nabla_{\! \tau \,} w|^2,
\qquad 
w \, \nabla_{\! \tau \,} v = w \, \nabla_{\! \tau \,} w  
\quad  {{\textrm{a.e.\ on}~\Gamma_{\! \nu}}}.
$$
As a consequence of position~(\ref{Dw}), the first integral in~(\ref{zero}) is non-negative, and hence
\begin{equation}\label{obvious}
\int_{\Gamma_{\! \nu}}
\big(
a_2 \nabla_{\! \tau \,} w
+
w \, \nabla_{\! \tau \,} a_2
\big)
\, \nabla_{\! \tau \,} v
\, ds
+
\int_{\Gamma_{\! \nu}}
a_0 \, vw \, ds
\le
0
.
\end{equation}
Let us focus on the product $vw$. By definition, the function $w = (v - k)^+$ is either positive or zero. When $w$ vanishes, the product~$vw$ obviously vanishes. When $w$ is positive, instead, $v - k$ is also positive and therefore $v > k > 0$. In conclusion, we have $vw \ge 0$ on~$\Gamma_{\! \nu}$. Since $a_0 \ge 0$, it follows that the last integral in~(\ref{obvious}) is non-negative, and therefore we arrive at
$$
\int_{\Gamma_{\! \nu}}
a_2 \, |\nabla_{\! \tau \,} w|^2 \, ds
\le
\int_{\Gamma_{\! \nu}}
|\nabla_{\! \tau \,} a_2|
\,
|\nabla_{\! \tau \,} w| \, w \, ds
.
$$%
Let $C = \max\{\, 1, \, M/\lambda_2 \,\}$, where $M,\lambda_2$ are as in \eqref{lambda-M}. In particular, we are assuming $\lambda_2 > 0$ (ellipticity). With this notation, we may write
\begin{equation}\label{C}
\int_{\Gamma_{\! \nu}}
|\nabla_{\! \tau \,} w|^2 \, ds
\le
C
\int_{\Gamma_{\! \nu}}
|\nabla_{\! \tau \,} w| \, w \, ds
.
\end{equation}
We aim to replace the domain of integration $\Gamma_{\!\nu}$ with the set $\tilde \Gamma = \bigcup\limits_{i = 1}^{N_1} \tilde \Gamma_i$, where
\begin{align*}
\tilde \Gamma_i
&=
\{\, {\bf x} \in \Gamma_{\! i}
\mid
0 < w({\bf x}) < \max\limits_{\Gamma_{\! i}} w \, \}
\\
\noalign{\medskip}
&=
\{\, {\bf x} \in \Gamma_{\! i}
\mid
k < v({\bf x}) < \max_{\Gamma_{\! i}} v \,\},\
i = 1, \dots, N_1
.
\end{align*}
By definition, $\tilde \Gamma_i$ is a (possibly empty) relatively open subset of~$\Gamma_i$. Furthermore, the equality $\Vert \nabla_{\!\tau\,} w \Vert_{L^2(\tilde \Gamma_{\! i})} = \Vert \nabla_{\!\tau\,} w \Vert_{L^2(\Gamma_{\! i})}$ holds because $\nabla_{\!\tau\,} w$ vanishes almost everywhere in~$\Gamma_i \setminus \tilde \Gamma_i$ (see~\cite[Lemma~7.7]{GT}). Since $\nabla_{\! \tau \,} w = 0$ on~$\Gamma_i$ for $i > N_1$, inequality (\ref{C}) implies
\begin{equation}\label{CN1}
\int_{\tilde \Gamma}
|\nabla_{\! \tau \,} w|^2 \, ds
\le
C
\int_{\tilde \Gamma}
|\nabla_{\! \tau \,} w| \, w \, ds
.
\end{equation}
In order to estimate the last integral, using the Cauchy--Schwarz inequality in~$\tilde \Gamma$ we get
$$
\int_{\tilde \Gamma} |\nabla_{\! \tau \,} w| \, w \, ds
\le
\Vert \nabla_{\! \tau \,} w
\Vert_{L^2(\tilde \Gamma)} \, \Vert w \Vert_{L^2(\tilde \Gamma)} .
$$%
By~\eqref{Gamma_1}, $\nabla_{\! \tau \,} w$ cannot vanish identically in~$\tilde \Gamma_1$, hence $\Vert \nabla_{\! \tau \,} w \Vert_{L^2(\tilde \Gamma)} > 0$. Thus, the inequality above and~(\ref{CN1}) imply
\begin{equation}\label{CS}
\Vert \nabla_{\! \tau \,} w \Vert_{L^2(\tilde \Gamma)}
\le
C \, \Vert w \Vert_{L^2(\tilde \Gamma)}
.
\end{equation}
Let us estimate the last term. For every  $i = 1, \dots, N_1$ we have $k > \mu_0^+ \ge \min\limits_{\Gamma_{\! i}} v$, and consequently $\min\limits_{\Gamma_{\! i}} w = 0$. Hence by relation \eqref{PoincPesataLemmaBISAntonio} we get $\Vert w \Vert^2_{L^2(\tilde \Gamma_{\! i})} \le |\tilde \Gamma_i|^2 \, \Vert \nabla_{\! \tau \,} w \Vert^2_{L^2(\tilde \Gamma_i)}$, and by summation over~$i = 1, \dots, N_1$
$$
\Vert w \Vert_{L^2(\tilde \Gamma)}
\le
|\tilde \Gamma|
\,
\Vert \nabla_{\! \tau \,} w \Vert_{L^2(\tilde \Gamma)}
.
$$%
This and~(\ref{CS}) imply $C \, |\tilde \Gamma| \ge 1$. By contrast,
we now check that $|\tilde \Gamma| \searrow 0$ as $k \nearrow \mu_1$.
To see this, it suffices to observe that when $k$ increases, the
corresponding set $\tilde \Gamma = \tilde \Gamma(k)$ describes a
decreasing family of open subsets of~$\Gamma_{\! \nu}$ with finite
Hausdorff measure, and by the continuity of the measure we have
\begin{align*}
\lim_{k \nearrow \mu_1} |\tilde \Gamma(k)|
=
\Bigg|
\bigcap_{k \in (\mu_0^+ \! , \, \mu_1)}
\tilde \Gamma(k)
\,\Bigg|
=
|\emptyset|
=
0
,
\end{align*}
which is a contradiction. Thereafter, the unique function~$v \in V_0$
satisfying~(\ref{zero}) for every $w \in V_0$ is the null function.
\end{proof}
\subsection{Existence of a solution}
Let us now turn our attention to the question concerning the existence of solutions to the homogeneous problem \eqref{homogeneous}. To this aim, we have to recall  that the composite Sobolev space $V_0$ introduced in \eqref{DefiSpaceFunctionV} was endowed with the norm  \eqref{NormOnV0}. We first start with this 
\begin{lemma}\label{LemmaBilinearFormAdHoc}
Let Assumption \ref{MainAssumption} be complied. Suppose that  $0<a_2\in W^{1,\infty}(\Gamma_{\! \nu})$, $a_0\in L^\infty(\Gamma_{\! \nu})$ and for\/ $V_0$ as in \eqref{DefiSpaceFunctionV} let us consider the bilinear form  ${\mathfrak B} \colon V_0
\times V_0 \to \mathbb R$ given by
$$
{\mathfrak B}(v,w)
=
\int_\Omega
Dv \cdot Dw \, d {\bf x}
+
\int_{\Gamma_{\! \nu}}
a_2 \, \nabla_{\!\tau\,} v \, \nabla_{\!\tau\,} w \, ds
+
\int_{\Gamma_{\! \nu}}
(\nabla_{\!\tau\,} a_2 \, \nabla_{\!\tau\,} v) w \, ds
+
\int_{\Gamma_{\! \nu}}
a_0 \, v w\, ds.
$$%
Then, for $\lambda_0,\lambda_2$ and $M$ defined in \eqref{lambda-M} the bilinear form 
$$
{\mathfrak B}_{\sigma_0}(v,w)
:=
\begin{cases}
{\mathfrak B}(v,w) + \sigma_0 \int_{\Gamma_{\! \nu}} vw \, ds & \textrm{ if } \sigma_0 >0, \\
{\mathfrak B}(v,w)  & \textrm{ if } \sigma_0 \leq 0,
\end{cases}\quad \text{where} \quad \sigma_0=\frac{\, M^2 \,}{2\lambda_2}-\lambda_0,
$$%
is continuous and coercive.
\begin{proof} Part 1. Continuity. From the definition of ${\mathfrak B}$, we directly have by means of the the Cauchy--Schwarz inequality and \eqref{lambda-M}
\begin{align}\label{FirstStepContinuityL}
\nonumber
|{\mathfrak B}(v,w)|&\leq \int_\Omega
|Dv \cdot Dw| \, d {\bf x}
+\Lambda_2 
\int_{\Gamma_{\! \nu}}
 |\nabla_\tau v || \nabla_\tau w|\, ds
\\
&\quad+M
\nonumber
\int_{\Gamma_{\! \nu}}
|\nabla_{\!\tau\,} v| \, |w| \, ds +\Lambda_0
\int_{\Gamma_{\! \nu}}
|v| \, |w|\, ds
\\
\noalign{\medskip}
\nonumber
&\leq \lVert Dv \rVert_{L^2(\Omega)} \, \lVert Dw \rVert_{L^2(\Omega)}  +\Lambda_2 \, \lVert \nabla_{\!\tau\,} v \rVert_{L^2(\Gamma_{\!\nu})} \, \lVert \nabla_{\!\tau\,} w \rVert_{L^2(\Gamma_{\! \nu})}\\
\noalign{\bigskip}
&\quad +M \, \lVert \nabla_{\!\tau\,} v \rVert_{L^2(\Gamma_{\! \nu})} \, \lVert  w \rVert_{L^2(\Gamma_{\! \nu})}
+\Lambda_0 \, \lVert  v \rVert_{L^2(\Gamma_{\! \nu})} \, \lVert  w \rVert_{L^2(\Gamma_{\! \nu})}.
\end{align}
On the other hand, applications of Lemma \ref{PoincPesataLemma} lead to $\Vert v \Vert_{L^2(\Gamma_{\! \nu})} \le L \, \lVert D v\rVert_{L^2(\Omega)}$ and $\Vert w \Vert_{L^2(\Gamma_{\! \nu})} \le L \, \lVert D w\rVert_{L^2(\Omega)}$. These inequalities in conjunction with \eqref{FirstStepContinuityL} yield for some positive $k$
\begin{equation*}
|{\mathfrak B}(v,w)|\leq k \, \lVert v \rVert_{V_0} \, \lVert w \rVert_{V_0}.
\end{equation*}
Further,  for all $v,w \in V_0$ we similarly have
\begin{equation*}
\int_{\Gamma_\nu}vwds\leq \Vert v \Vert_{L^2(\Gamma_{\! \nu})} \Vert w \Vert_{L^2(\Gamma_{\! \nu})}.
\end{equation*}
As an immediate consequence of this, by invoking again Lemma \ref{PoincPesataLemma}, as well as recalling definition \eqref{NormOnV0}, it holds that
\begin{equation*}
|{\mathfrak B}_{\sigma_0}(v,w)|\leq |{\mathfrak B}(v,w)|+\Big|\sigma_0\int_{\Gamma_\nu}vwds\Big|\leq L_1 \, \lVert v \rVert_{V_0} \, \lVert w \rVert_{V_0} \quad \textrm{for all } v,w \in V_0,
\end{equation*}
where $L_1$ is some positive constant. Hence ${\mathfrak B}_{\sigma_0}$ is continuous.

\goodbreak Part 2. Coercivity. Again in light of positions \eqref{lambda-M} we directly have
\begin{align}\label{InequalityForCoherci}
\nonumber
&\int_{\Gamma_{\! \nu}}
a_2 \, |\nabla_{\!\tau\,} v|^2\, ds
+
\int_{\Gamma_{\! \nu}}
(\nabla_{\!\tau\,} a_2 \, \nabla_{\!\tau\,} v) \, v \, ds+
\int_{\Gamma_{\! \nu}}
a_0 \, v^2 \, ds
\\
\noalign{\medskip}
&\ge
\lambda_2
\int_{\Gamma_{\! \nu}}
|\nabla_\tau v|^2 \, ds
-
M
\int_{\Gamma_{\! \nu}}
|\nabla_{\!\tau\,} v| \, |v| \, ds+\lambda_0\int_{\Gamma_{\! \nu}}
 v^2 \, ds
.
\end{align}
Now if $M = 0$ then $\sigma_0 \le 0$ and therefore
$$
{\mathfrak B}_{\sigma_0}(v,v)
=
{\mathfrak B}(v,v)
\ge
\int_\Omega |Dv|^2 \, d {\bf x}
+
\lambda_2
\int_{\Gamma_{\! \nu}}
|\nabla_{\!\tau\,} v|^2 \, ds
\geq L_2 \, \lVert v \rVert^2_{V_0}
$$%
where $L_2 = \min\{\, 1,\lambda_2 \,\} > 0$, hence ${\mathfrak B}_{\sigma_0}$ is coercive. If, instead, $M > 0$, then we let $\varepsilon = \lambda_2/(2M)$ in the Young inequality and obtain $|\nabla_{\!\tau\,} v| \, | v| \le \varepsilon \, |\nabla_{\!\tau\,} v|^2 + \frac1{4\varepsilon} \, v^2$. By plugging such an estimate into \eqref{InequalityForCoherci}, we arrive at
$$
{\mathfrak B}(v,v)
\ge
\int_\Omega |Dv|^2 \, d {\bf x}
+
\frac{\, \lambda_2 \,}2
\int_{\Gamma_{\! \nu}}
|\nabla_{\!\tau\,} v|^2 \, ds
-
\Big(\frac{\, M^2 \,}{2\lambda_2}-\lambda_0\Big)
\int_{\Gamma_{\! \nu}}
v^2 \, ds
$$%
and for $L_3=\min\{\, 1,\frac{\lambda_2}2 \,\} > 0$ we may write
$$
{\mathfrak B}_{\sigma_0}(v,v)
\ge
\int_\Omega |Dv|^2 \, d {\bf x}
+
\frac{\, \lambda_2 \,}2
\int_{\Gamma_{\! \nu}}
|\nabla_{\!\tau\,} v|^2 \, ds
\geq L_3 \, \lVert v \rVert^2_{V_0}
.
$$%
Hence ${\mathfrak B}_{\sigma_0}$ is coercive also in this case, and the proof is complete.
\end{proof}
\end{lemma}
\begin{lemma}[Unique solvability of the homogeneous problem]\label{LemmaExistence}
Let Assumption~\ref{MainAssumption} be complied. Suppose that\/ $0<a_2 \in W^{1,\infty}(\Gamma_{\! \nu})$ and $0\leq a_0 \in L^\infty(\Gamma_{\! \nu})$. Then problem \eqref{homogeneous} admits a unique weak solution.
\begin{proof}
We rely on the properties of the bilinear form ${\mathfrak B}_{\sigma_0}(v,w)$ introduced in Lemma \ref{LemmaBilinearFormAdHoc}. For $\sigma_0 \le 0$, equation~(\ref{weak}) may be rewritten as ${\mathfrak B}_{\sigma_0}(v,w) = {\cal L}w$, and existence and uniqueness of the solution~$v$ follow immediately from the Lax-Milgram theorem. If, instead, $\sigma_0 > 0$, then we invoke the Fredholm alternative \cite[Theorem~5.3]{GT} (see also \cite[Section~D.5]{Evans-2010-PDEs}). To be precise, solving equation~(\ref{weak}) is equivalent to finding $v \in V_0$ such that
\begin{equation}\label{LaxMilgramEquationModified}
{\mathfrak B}_{\sigma_0}(v,w)
=
{\cal L} w
+
\sigma_0 \int_{\Gamma_{\! \nu}} vw \, ds \quad \forall \, w \in V_0.
\end{equation}
By the Lax-Milgram theorem we may define the linear, continuous operator $T \colon V_0^* \allowbreak \to V_0$ given by $T(L) \allowbreak = v_L$, where $v_L$ is the unique solution of the equation ${\mathfrak B}_{\sigma_0}(v,w) = Lw$ in the unknown~$v$. For every $v \in V_0$ we also define the linear, continuous operator $I_v \in V_0^*$ given by
$$
I_{v \,} w = \int_{\Gamma_{\! \nu}} vw \, ds
.
$$%
Consequently, relation \eqref{LaxMilgramEquationModified} can be rewritten as
$$
v = T({\cal L}) + \sigma_0 \, T(I_v)
.
$$%
Existence and uniqueness of a solution~$v$ follow from the Fredholm alternative provided we ensure that
\begin{enumerate}
\renewcommand\labelenumi{$(\roman{enumi})$}
\item the homogeneous equation $v = \sigma_0 \, T(I_v)$ has
only the trivial solution;
\item the operator $K \colon v \mapsto T(I_v)$ is compact. 
\end{enumerate}
Condition $(i)$ was already established in Lemma \ref{LemmaUniqueness}. To check~$(ii)$, take a bounded sequence $(v_n)$ in~$V_0$. By Lemma~\ref{PoincPesataLemma}, the traces $v_{n|_{\Gamma_{\! \nu}}}$ are uniformly bounded in $H^1_0(\Gamma_{\! \nu})$, hence there exists a subsequence of functions still denoted by $v_n$ whose traces on~$\Gamma_{\! \nu}$ converge to some~$v_\infty$ in~$L^2(\Gamma_{\! \nu})$. Henceforth, the operators $I_{v_n}$ converge to $I_{v_\infty}$ in the norm of the dual space~$V_0^*$. Since the operator~$T$ is continuous, the sequence of functions $K(v_n) = T(I_{v_n})$ converges in~$V_0$. In conclusion, $K$~is a compact operator, and the lemma follows from the Fredholm theorem.
\end{proof}
\end{lemma}
As a consequence of all the above preparations, we finally can prove our main statement:

\begin{proof}[Proof of Theorem \ref{Maintheorem}]
The assumption on $\varphi$ implies by means of Lemma \ref{lifting} the existence of a lifting~$\tilde \varphi \in V$. Hence, the considerations presented in Remark \ref{RemarkFromInhToHom} convert problem \eqref{inhomogeneous} into a homogeneous problem having the form~\eqref{homogeneous}. Finally, Lemma~\ref{LemmaExistence} ensures its unique solvability.
\end{proof}
\appendix
\section{An application to the equilibrium problem for a prestressed membrane}\label{AppendixSection}
As announced in $\S$\ref{SectionIntro}, let us dedicate to the interplay between {\bf Equilibrium Problem 1} in \cite[Section~3.1]{ViglialoroGonzalezMurciaIJCM} and problem~\eqref{inhomogeneous}, exactly by presenting the formulation of the former in terms of the nomenclature employed in the present paper. We consider a plane domain~$\Omega$ satisfying Assumption~\ref{MainAssumption}, and we restrict to the case when $\Gamma_{\!\nu}$ is a curve parametrized by a vector-valued function $r(t) = (x(t),y(t))$ belonging to $C^2((t_0, \allowbreak t_1))\cap C^1([t_0,t_1])$ and satisfying $r'(t) \ne 0$ for all~$t$ in some bounded interval $[t_0,t_1]$ on the real line, with $t_0 < t_1$.

\paragraph{Expression of the Laplace-Beltrami operator.}
Let us consider a function $u \allowbreak \in C^2((\Omega \cup \Gamma_{\!\nu})$, and recall that the arc length along the curve~$\Gamma_{\!\nu}$ is given by
\begin{equation*}
s(t)
=
\int_{t_0}^t \sqrt{(x'(\xi))^2 + (y'(\xi))^2 \,} \, d\xi
.
\end{equation*}
Since $r'(t) \ne 0$, the function $s(t)$ admits a smooth inverse $t = t(s)$, so yielding (recall Remark \ref{RemarkLaplOperOnaCurve}) the following representation of the Laplace--Beltrami operator of $u$ along the one-dimensional manifold~$\Gamma_{\!\nu}$:
\begin{equation}\label{LaplBelOperator-u_ss}
\Delta_{\tau\,} u=u_{ss}
=
\frac{\, d^2 \,}{\, ds^2 \,}
\,
u(x(t(s)), \, y(t(s)))
.
\end{equation}
It is essential for our purposes to express $\Delta_{\tau\,} u$ in terms of the partial derivatives of~$u$ with respect to the Cartesian coordinates $x,y$, the outward derivative $u_\nu$, and the curvature $\kappa$ of the curve $\Gamma_{\!\nu}$, given by
\begin{equation}\label{kappa}
\kappa(t)
=
\frac{x'(t) \, y''(t)-y'(t) \, x''(t)}
{\, \big({(x'(t))^2 + (y'(t))^2}\big)^{3/2} \,}, \quad t \in (t_0,t_1).
\end{equation}
To this purpose, we will make use of this obvious identity
\begin{equation}\label{obviousBIS}
\frac{d}{\, dt \,}
=
\sqrt{(x'(t))^2 + (y'(t))^2 \,} \, \frac{d}{\, ds \,}
.
\end{equation}
Furthermore, for every point  $(x,y) = (x(t), y(t)) \in \Gamma_{\!\nu}$, and by introducing the tangent unit vector to $\Gamma_{\!\nu}$
\begin{equation}\label{tangetUnitvector}
\tau:=\tau(t)
=
\frac{\, (x'(t),y'(t)) \,}
{\, \sqrt{(x'(t))^2 + (y'(t))^2 \,} \,},\quad t \in (t_0,t_1),
\end{equation}
we let
\begin{align}
\nonumber
u_\tau
&=
\tau \cdot Du(x(t),y(t))
\\
\noalign{\bigskip}
\label{u_tau}
&=
\frac{\,
	x'(t) \, u_x(x(t),y(t)) + y'(t) \, u_y(x(t),y(t))
	\,}
{\sqrt{(x'(t))^2 + (y'(t))^2 \,}},\quad t \in (t_0,t_1),
\\
\noalign{\bigskip}
\nonumber
u_{\tau\tau}
&=
\tau \cdot D^2u(x(t),y(t)) \, \tau^T
\\
\noalign{\bigskip}
\nonumber
&=
\frac{\, (x'(t))^2 \, u_{xx}(x(t),y(t))
+
2 \, x'(t) \, y'(t) \, u_{xy}(x(t),y(t)) \,}
{(x'(t))^2 + (y'(t))^2}
\\
\noalign{\bigskip}
\label{u_tautau} 
&
\quad + \frac{\, (y'(t))^2 \, u_{yy}(x(t),y(t)) \,}
{(x'(t))^2 + (y'(t))^2},\quad t \in (t_0,t_1),
\end{align}
where $Du(x,y) = (u_x(x,y), \, u_y(x,y))$ is the gradient of~$u$, $D^2u(x,y)$ the Hessian matrix and $\tau^T$ the transposed of~$\tau$.  The following lemma shows the relation between  $u_{ss}$, $u_{\tau\tau}$ and the outward derivative of~$u$ on $\Gamma_{\nu}$, explicitly given by
\begin{equation}\label{u_nu}
u_\nu
=
\frac{\, x'(t) \, u_y(x(t),y(t)) - y'(t) \, u_x(x(t),y(t)) \,}
{\, \sqrt{(x'(t))^2 + (y'(t))^2 \,} \,}, \quad t \in (t_0,t_1).
\end{equation}
\begin{lemma}\label{equivalence}
	If $u \in C^2(\Omega \cup \Gamma_{\nu})$, then
	$u_{ss} = u_{\tau\tau} + \kappa(t) \, u_\nu$ on~$\Gamma_{\nu}$.
\begin{proof}
By computing the derivative in the right-hand side of $u_s = \frac{\, d \,}{\, ds \,} \, u(x(t(s)), \allowbreak\, y(t(s)))$, and taking~(\ref{obviousBIS}) into account, we find for all $t \in (t_0,t_1)$
\begin{equation}\label{u_s}
\sqrt{(x'(t))^2 + (y'(t))^2 \,} \, u_s = x'(t) \, u_x(x(t),y(t)) + y'(t) \, u_y(x(t),y(t))
,
\end{equation}
which in particular shows due to \eqref{u_tau} that $u_s = u_\tau$. Differentiating both sides of~(\ref{u_s}) with respect to~$t$ yields
\begin{align*}
&\frac{\, x'(t) \, x''(t) + y'(t) \, y''(t)}{\, \sqrt{(x'(t))^2 + (y'(t))^2 \,} \,}
\,
u_s
+
\big((x'(t))^2 + (y'(t))^2\big)
\,
u_{ss}
\\
\noalign{\medskip}
&=
(x'(t))^2 \, u_{xx}(x(t),y(t)) + 2 \, x'(t) \, y'(t) \, u_{xy}(x(t),y(t)) + (y'(t))^2 \, u_{yy}(x(t),y(t))
\\
\noalign{\bigskip}
&
\quad +x''(t) \, u_x(x(t),y(t))
+ y''(t) \, u_y(x(t),y(t)), \quad t \in (t_0,t_1).
\end{align*}
By dividing both sides by $(x'(t))^2 + (y'(t))^2$ and using relation~\eqref{u_tautau}, the last equality becomes, on $(t_0,t_1)$,
$$
\frac{\, x'(t) \, x''(t) + y'(t) \, y''(t) \,}
{((x'(t))^2 + (y'(t))^2)^\frac{3}{2}} \, u_s
+
u_{ss}
=
u_{\tau\tau}
+
\frac{\, x''(t) \, u_x(x(t),y(t))
+ y''(t) \, u_y(x(t),y(t)) \,}{(x'(t))^2 + (y'(t))^2}.
$$%
Finally, taking \eqref{kappa} and~(\ref{u_s}) into consideration, we arrive at
$$
u_{ss}
=
u_{\tau\tau}
+
\kappa(t)
\,
\frac{\, x'(t) \, u_y(x(t),y(t)) - y'(t) \, u_x(x(t),y(t)) \,}
{\, \sqrt{(x'(t))^2 + (y'(t))^2 \,} \,},\quad t \in (t_0,t_1),
$$%
and the lemma follows thanks to (\ref{u_nu}).
\end{proof}
\end{lemma} 

\paragraph{Setting the domain.}
In view of the example below we assume strict convexity of the curve~$\Gamma_{\!\nu}$, in the sense that $\kappa(t) > 0$ for all $t \in (t_0,t_1)$. We also suppose
$$
x(t_0)<0,\quad x(t_1)>0, \quad \textrm{and} \quad \textrm{$x'(t),y(t)>0$ for all $t \in (t_0,t_1)$.}
$$
By setting 
\begin{equation}\label{boundary}
\begin{cases}
\Gamma_{\nu} = \{\, (x(t),y(t)) \in \mathbb{R}^2 \mid t \in (t_0,t_1) \,\},\\
\noalign{\medskip}
\Gamma_{\!D}
=
\{\, (x(t_0),t)  \in \mathbb{R}^2 \mid 0 \le t \le y(t_0) \,\}
\\
\qquad \quad \cup
\{\, (t,0))  \in \mathbb{R}^2 \mid t \in (x(t_0),x(t_1)) \,\}
\\
\qquad \quad \cup
\{\, (x(t_1),t)  \in \mathbb{R}^2 \mid 0 \le t \le y(t_1) \,\},
\end{cases}
\end{equation} 
we consider the planar, non-empty, open and bounded domain $\Omega$, uniquely determined by the condition $\partial \Omega = \Gamma_{\nu} \cup \Gamma_{\!D}$. In these circumstances, the equilibrium of a prestressed membrane obeys an elliptic equation in $\Omega$ for the unknown $u \colon \Omega \to \mathbb{R}$ (whose graph represents the shape of the membrane) endowed with a classical Dirichlet condition on $\Gamma_{\!D}$ (rigid boundaries, with their own stiffness), which basically fixes the shape of the membrane on the portion $\Gamma_{\!D}$, and a Ventcel-type one on $\Gamma_{\nu}$ (non-rigid boundaries, {\textit{without any stiffness}}), idealizing the physical equilibrium for cable elements. 
\paragraph{The mixed boundary value problem.}
As discussed in \cite[Section 2.2.2]{ViglialoroGonzalezMurciaIJCM}, 
the introduction of a cable boundary as a structural element providing equilibrium to a membrane requires a restriction on its shape, which is modeled by the convexity assumption $\kappa(t) > 0$, and the so-called \textit{cable-membrane compatibility equation} $u_{\tau\tau} = 0$ along $\Gamma_\nu$ (see \cite[relation (5)]{ViglialoroGonzalezMurciaIJCM} for a particular parametrization of $\Gamma_\nu$). Recalling Lemma~\ref{equivalence} and relation~\eqref{LaplBelOperator-u_ss}, we have
$$
u_{\tau \tau}=0 \quad \Leftrightarrow \quad u_{\nu}-\frac{1}{\, \kappa(x) \,} \, \Delta_{\tau\,} u=0.
$$
Now consider the most simplified case of {\bf Equilibrium Problem 1} in \cite[Section~3.1]{ViglialoroGonzalezMurciaIJCM}, which is obtained when $\mbox{\boldmath$\sigma$} =c {{\bf 1}}$. The problem corresponds to the equilibrium of a membrane tensioned with the same stress $c>0$ in the two main orthogonal directions and not exposed to external loads. The comments above make such a problem read as follows: fixed a sufficiently regular function $\varphi \colon \Gamma_{\!D} \to \mathbb R$, find $u$ such that 
\begin{equation*}
\begin{cases}
\Delta u = 0 &\mbox{in $\Omega$;}
\\
u = \varphi &\mbox{on $\Gamma_{\!D}$;}
\\
u_\nu -\frac1{\, \kappa(x) \,} \, \Delta_{\tau\,} u  = 0 &\mbox{on $\Gamma_{\nu}$.}
\end{cases}
\end{equation*}
This is manifestly a special case of problem \eqref{inhomogeneous}, which is well-posed by virtue of Theorem \ref{Maintheorem}.
\paragraph{Example of an explicit solution.}
For $r(t)=(x(t),y(t))$ with 
\begin{equation*}
(x(t),y(t))
=
\frac{(t,1)}{\, \left(2-\rho(t)\right)^\frac23 \left(1+\rho(t)\right)^\frac13 \,}, \quad \rho(t) = \sqrt{1+t^2 \,}, \quad t\in (-1,1),
\end{equation*}
let us define $\Gamma_{\!\nu},\Gamma_{\!D}$ as in~\eqref{boundary}, so that the domain $\Omega$ is the open and bounded set with boundary $\partial \Omega = \Gamma_{\nu} \cup \Gamma_{\!D}$. By differentiation we find
$$
\begin{cases}
x'(t) = y(t) + t \, y'(t),
\\
\noalign{\medskip}
y'(t) = t \, (2-\rho(t))^{-\frac53} \, (1+\rho(t))^{-\frac43}
\end{cases}
$$%
and
$$
\begin{cases}
x''(t) = 2y'(t) + t \, y''(t)
\\
\noalign{\medskip}
y''(t) =(\rho(t))^{-1} \, (2-\rho(t))^{-\frac83} \, (1+\rho(t))^{-\frac73} \, (1 - \rho + 2\rho^3)
\end{cases}
$$%
In order to compute $\kappa(t)$ from~\eqref{kappa}, we prepare the equalities
\begin{align*}
x'(t) \, y''(t) - y'(t) \, x''(t)
&=
y(t) \, y''(t) - 2(y'(t))^2
\\
&=
(\rho(t))^{-1} \, (2-\rho(t))^{-\frac{10}3} \, (1+\rho(t))^{-\frac53}
\end{align*}
and
\begin{align*}
(x'(t))^2 + (y'(t))^2
&=
y^2(t) + 2t \, y(t) \, y'(t) + \rho^2(t) \, (y'(t))^2
\\
\noalign{\smallskip}
&=
2\rho(t) \, (2-\rho(t))^{-\frac{10}3} \, (1 + \rho(t))^{-\frac53}
\end{align*}
which imply
$$
\kappa(t)
=
\frac{\, (2 - \rho(t))^\frac53 \, (1 + \rho(t))^\frac56}
{2^\frac32 \, \rho^\frac52(t)}
> 0, \quad t \in (-1,1).
$$%
The tangent unit vector to $\Gamma_\nu$, whose expression is found in~\eqref{tangetUnitvector}, reads as
$$
\tau =
\frac{\, \big( y(t) + t \, y'(t), \, y'(t) \big) \,}
{\, (2\rho(t))^\frac12  \, (2 - \rho(t))^\frac53 \, (1 + \rho(t))^\frac56 \,}
=
\frac{(1+\rho(t), \, t)}
{\, \sqrt{2\rho(t) \, (1 + \rho(t)) \,} \,}
,
\quad t \in (-1,1).
$$%
Additionally, let us give on $\Omega$ the smooth function $u({\bf x})=u(x,y)=x^3-3xy^2$. We obviously have
\[
D^2u(x,y)=6
\begin{pmatrix}
x & -y \\
-y & -x 
\end{pmatrix}
\quad \forall \, {\bf x} \in \Omega,
\]
and therefore
$$
D^2u(x(t),y(t))=6 \, y(t)
\begin{pmatrix}
t & -1 \\
-1 & -t 
\end{pmatrix}
\!,
\quad t \in (-1,1),
$$%
 from the expression of $\tau$ it is entailed that (recall \eqref{u_tautau}) $u_{\tau\tau}=0$ for all $t$ in $(-1,1)$; henceforth,  we conclude that $u$ classically solves
\begin{equation*}
\begin{cases}
\Delta u = 0 &\mbox{in $\Omega$;}
\\
\noalign{\bigskip}
u(x,0) = x^3&\mbox{for $x \in (x(-1),x(1))$;}
\\
\noalign{\bigskip}
u(x(-1),y) = \displaystyle -\frac1{\, 2 \,}-\frac{1}{\, \sqrt{2 \,} \,}+\frac{3 y^2}{\, \left(2-\sqrt{2 \,}\right)^{2/3} \, \left(1+\sqrt{2}\right)^{1/3} \,}&\mbox{for $y \in [x(-1),y]$;}\\
u(x(1),y) = \displaystyle \frac1{\, 2 \,}+\frac{1}{\, \sqrt{2 \,} \,}-\frac{3 y^2}{\, \left(2-\sqrt{2}\right)^{2/3} \left(1+\sqrt{2 \,}\right)^{1/3} \,}&\mbox{for $y \in [x(1),y]$;}\\
u_{\tau\tau}=\displaystyle u_\nu -\frac1{\, \kappa(t) \,} \, \Delta_{\tau\,} u  = 0 &\mbox{on $\Gamma_{\nu}$.}
\end{cases}
\end{equation*}
\pdfbookmark{Acknowledgements}{Acknowledgements}
\subsubsection*{Acknowledgements} The authors are members of the Gruppo Nazionale per l'Analisi Matematica, la Probabilit\`a e le loro Applicazioni (GNAMPA) of the Istituto Na\-zio\-na\-le di Alta Matematica (INdAM) and are  partially supported by the research projects \textit{Integro-differential Equations and Non-Local Problems}, funded by Fondazione di Sardegna (2017). GV is partially supported by MIUR (Italian Ministry of Education, University and Research) Prin 2017 \textit{Nonlinear Differential Problems via Variational, Topological and Set-valued Methods} (Grant Number: 2017AYM8XW).

\bibliographystyle{abbrv}
\end{document}